\documentclass[11pt]{amsart}
\usepackage{amsmath,amssymb,latexsym, dsfont}
\usepackage[small]{caption}
\usepackage{graphicx,color,mathrsfs,tikz}
\usepackage{subfigure,color,bm}
\usepackage{cite}
\usepackage[colorlinks=true,urlcolor=blue,
citecolor=red,linkcolor=blue,linktocpage,pdfpagelabels,
bookmarksnumbered,bookmarksopen]{hyperref}
\usepackage[italian,english]{babel}
\usepackage{units}
\usepackage{enumitem}
\usepackage[left=2.5cm,right=2.5cm,top=2.9cm,bottom=2.9cm]{geometry}
\usepackage[hyperpageref]{backref}
\usepackage{fancyhdr, dsfont}
\usepackage{tikz}
\usepackage{hyperref,cleveref}

\usepackage{amscd}
\usepackage{esint}

\newcommand{\eps}{\varepsilon}

\theoremstyle{plain}

\newtheorem{thm}{Theorem}[section]
\newtheorem*{thm*}{Theorem}
\theoremstyle{plain}
\newtheorem{lem}[thm]{Lemma}
\newtheorem{pro}{Proposition}[section]
\newtheorem{cor}[thm]{Corollary}

\theoremstyle{definition}

\newtheorem{remark}{Remark}[section]

\newcommand{\mN}{\mathbb{N}}
\newcommand{\mS}{\mathbb{S}}
\newcommand{\mR}{\mathbb{R}}

\newcommand{\ro}{\mathbb{R}}

\newcommand{\mcal}[1]{\mathcal{#1}}
\newcommand{\md}[1]{\left|#1 \right|}
\newcommand{\nrm}[1]{\left\| #1 \right\|}
\newcommand{\bct}[1]{\left(#1\right)}

\newcommand{\Sb}[1]{\left \lbrace #1\right\rbrace}

\newcommand{\hu}{\hat u}

\newcommand{\C}{{\mathcal A}}
\newcommand{\rr}{\tau}
\newcommand{\R}{\mathbb{R}}

\numberwithin{equation}{section} \allowdisplaybreaks

        \title[The Caffarelli-Kohn-Nirenberg  inequalities for radial functions]{The Caffarelli-Kohn-Nirenberg  inequalities for radial functions}

        \author[A. Mallick]{Arka Mallick}
        \address[A. Mallick]{Department of Mathematics \newline\indent
	IISc, Bengaluru,India}
\email{arkamallick@iisc.ac.in}

        \author[H.-M. Nguyen]{Hoai-Minh Nguyen}
  \address[H.-M. Nguyen]{Laboratoire Jacques Louis Lions, \newline\indent
 Sorbonne Universit\'e\newline\indent
	Paris, France}
\email{hoai-minh.nguyen@sorbonne-universite.fr}

\begin{document}

\begin{abstract}
We establish the {\it full range}  of the Caffarelli-Kohn-Nirenberg inequalities for radial functions in the Sobolev and the fractional Sobolev spaces  of order $0 < s \le 1$. In particular, we show that the range of the parameters for radial functions is strictly larger than the one without symmetric assumption. Previous known results reveal 
only  some special ranges of parameters even in the case $s=1$. Our proof is new and can be easily adapted to other contexts.  Applications on compact embeddings are also mentioned. 
\end{abstract}

\maketitle
\tableofcontents


\section{Introduction}
Let $d\geq 1$, $p\geq 1$, $q\geq 1$, $\tau \ge 1$, $0 <  a\leq 1$, $\alpha,\beta, \gamma \in \ro$ be such that 
\begin{equation}\label{positive-1}
\frac{1}{\tau}+\frac{\gamma}{d}, \ \frac{1}{p}+\frac{\alpha}{d}, \ \frac{1}{q}+\frac{\beta}{d}>0, 
\end{equation}  
and  the following balance law  holds 
\begin{equation}\label{balance-law1}
\frac{1}{\tau}+\frac{\gamma}{d}= a\bct{ \frac{1}{p}+\frac{\alpha-1}{d}}+ (1-a)\bct{ \frac{1}{q}+\frac{\beta}{d}}. 
\end{equation}
Define $\sigma$ by 
\begin{equation}\label{def-alpha-gamma}
\gamma = a\sigma+(1-a)\beta.
\end{equation}
Assume that 
\begin{equation}\label{coucou1}
0\leq \alpha-\sigma
\end{equation}
and 
\begin{equation}\label{coucou2}
\alpha-\sigma\leq 1 \text{ if } \frac{1}{\tau}+\frac{\gamma}{d}= \frac{1}{p}+\frac{\alpha-1}{d}.
\end{equation}

Caffarelli, Kohn,   and Nirenberg \cite{CKN} (see also \cite{CKN1}) established the following famous Caffarelli, Kohn  and Nirenberg (CKN) inequalities, for $u \in C_c^1(\ro^d)$,   
\begin{equation}\label{CKN-original}
\nrm{|x|^\tau u}_{L^\tau(\ro^d)} \leq C\nrm{|x|^\alpha \nabla u}^a_{L^p(\ro^d)} \nrm{|x|^\beta u}^{1-a}_{L^q(\ro^d)}, 
\end{equation}
for some positive constant $C$ independent of $u$. Quite recently, the full range of the CKN inequalities has been derived   by Nguyen and Squassina \cite{Ng-Squa1} for the fractional Sobolev spaces $W^{s, p}(\mR^d)$ with $0 < s < 1$ and $p > 1$. More precisely,   let $d\geq 1$, $0<s<1$, $p>1$, $q\geq 1$, $\tau \ge 1$, $0 <  a \leq 1$,  and $\alpha_1, \alpha_2,\beta, \gamma \in \ro$. Set $\alpha = \alpha_1+\alpha_2 $ and define $\sigma$ by  \eqref{def-alpha-gamma}. Assume that 
\begin{align}\label{balance-law2}
\frac{1}{\tau}+\frac{\gamma}{d}= a\bct{\frac{1}{p}+ \frac{\alpha-s}{d}}+(1-a)\bct{\frac{1}{q}+\frac{\beta}{d}},
\end{align}
and the following conditions hold 
\begin{equation}\label{cond-s-1}
0 \leq \alpha-\sigma
\end{equation}
and 
\begin{equation}\label{cond-s-2}
\alpha-\sigma\leq s \text{ if } \frac{1}{\tau}+\frac{\gamma}{d}= \frac{1}{p}+\frac{\alpha-s}{d}.
\end{equation}
Nguyen and Squassina \cite[Theorem 1.1]{Ng-Squa1} proved, for some positive constant $C$,  
\begin{itemize}
\item[$i)$] if $\frac{1}{\tau}+\frac{\gamma}{d}>0$, then for all $u \in C^1_c(\R^d)$, it holds 
\begin{align}\label{CKN-fractional1}
\nrm{|x|^\gamma u }_{L^\tau(\ro^d)} \leq C \bct{\int_{\ro^d}\int_{\ro^d} \frac{|u(x)-u(y)|^p|x|^{\alpha_1p}|y|^{\alpha_2p}}{|x-y|^{d+sp}}dx dy }^\frac{a}{p}\nrm{|x|^\beta u}^{1-a}_{L^q(\ro^d)},
\end{align}

\item[$ii)$]  if $\frac{1}{\tau}+\frac{\gamma}{d} < 0$, then for all $u \in C^1_c(\R^d \setminus \{0 \})$, it holds 
\begin{align}\label{CKN-fractional2}
\nrm{|x|^\gamma u }_{L^\tau(\ro^d)} \leq C \bct{\int_{\ro^d}\int_{\ro^d} \frac{|u(x)-u(y)|^p|x|^{\alpha_1p}|y|^{\alpha_2p}}{|x-y|^{d+sp}}dx dy }^\frac{a}{p}\nrm{|x|^\beta u}^{1-a}_{L^q(\ro^d)},
\end{align}
\end{itemize}
In the case $\frac{1}{\tau}+\frac{\gamma}{d}  = 0$, a $\log$-correction is required,  and the conditions \eqref{cond-s-1} and \eqref{cond-s-2} are replaced by 
\begin{equation}\label{cond-s-3}
0 \le \alpha - \sigma \le s. 
\end{equation}
Denote $B_R$ the open ball centered at the origin with radius $R$. Assume {\it additionally}  that  $\tau > 1$.   Nguyen and Squassina \cite[Theorem 3.1]{Ng-Squa1} showed that there exists a positive constant $C$ such that for all $u\in C_c^1(\ro^d)$ and for all $R_1, R_2 > 0$,  we have  
\begin{itemize}
\item[(i)] if $\frac{1}{\tau}+\frac{\gamma}{d}=0$ and $\mathrm{supp}\ u \subset B_{R_2}$, then 
\begin{multline}\label{CKN-limiting-1}
\bct{\int_{\ro^d} \frac{|x|^{\gamma \tau}}{\ln^{\tau} (2R_2/|x|)}|u|^{\tau}dx}^\frac{1}{\tau} \\[6pt]\leq C\bct{\int_{\ro^d}\int_{\ro^d} \frac{|u(x)-u(y)|^p|x|^{\alpha_1p}|y|^{\alpha_2p}}{|x-y|^{d+sp}}dx dy }^\frac{a}{p}  \nrm{|x|^\beta u}^{1-a}_{L^q(\ro^d)}, 
\end{multline}
 \item[(ii)] if $\frac{1}{\tau}+\frac{\gamma}{d}=0$, and $\mathrm{supp} \ u \cap B_{R_1}= \emptyset$, then 
 \begin{multline}\label{CKN-limiting-2}
\bct{\int_{\ro^d} \frac{|x|^{\gamma \tau}}{\ln^{\tau} (2|x|/R_1)}|u|^{\tau}dx}^\frac{1}{\tau} \\[6pt]\leq C \bct{\int_{\ro^d}\int_{\ro^d} \frac{|u(x)-u(y)|^p|x|^{\alpha_1p}|y|^{\alpha_2p}}{|x-y|^{d+sp}}dx dy }^\frac{a}{p} \nrm{|x|^\beta u}^{1-a}_{L^q(\ro^d)}. 
\end{multline}
\end{itemize}

Note that the conditions $\frac{1}{p}+\frac{\alpha}{d}, \ \frac{1}{q}+\frac{\beta}{d}>0$ are not required in these inequalities. In the case $a=1$ and $1/ \tau + \gamma / d > 0$, several special ranges of parameters were previously derived in \cite{FS08, MS02, AB17}. These works are partially motivated from new characterizations of Sobolev spaces using non-local, convex functionals proposed by Bourgain, Brezis, and Mironescu \cite{BBMAnother} (see also \cite{BrHow}). Related characterizations of Sobolev spaces with non-local, non-convex functionals can be found in \cite{NgSob1, Bour-Ng-Sobolev, BrNg18} and the references therein. The proof given in \cite{Ng-Squa1} (see also \cite{Ng-Squa2}) is new. It is based on the dyadic decomposition of the real space, Gagliardo-Nirenberg's inequalities for annulus,  and a trick on  summation processes to bring the information from a family of  annulus to the whole space. Combining these ideas with the techniques in \cite{Ng-Ineq}, which are used to prove new Sobolev's inequalities, we established the full range of Coulomb-Sobolev inequalities \cite{Ng-Mallick2}. In the case $s=1$,  inequality~\eqref{CKN-original} also holds in the case $1/ \tau  + \gamma/d < 0$,  and similar results as in \eqref{CKN-limiting-1} and \eqref{CKN-limiting-2} are valid in the case $1/ \tau  + \gamma/d = 0$. We present these results in \Cref{sect-s=1} (see \Cref{thm3} and \Cref{thm4}).

\medskip 
In this paper, we investigate the CKN inequalities for radial functions. We show that  the previous results also hold for some negative range of $\alpha - \sigma$ (compare with \eqref{coucou1} and \eqref{cond-s-1}). 
The  fact that the range of the parameters of  a family of inequalities can be larger when a symmetry condition is imposed is a well-known phenomenon, e.g.,  in the context of Stein-Weis inequalities \cite{Rubin83,DDD12} and  Coulomb-Sobolev inequalities \cite{BGO, BGMMV}. 
Various compactness results can be established using the extended range and are useful in the proof of the existence of minimizers of variational problems. Also, these compactness results play important roles in the analysis of various interesting physical phenomena, see,  e.g.,  \cite{Strauss77, BL1, Lieb77, Lions82}, and the references therein. It is quite surprising that  very few results have been known for the extended range of the CKN inequalities for radial functions. The goal of this paper is to completely fill this gap  for $0 <  s \le 1$.

\medskip 
We first concentrate on the setting of the fractional Sobolev spaces.  The following notation is used.   For $p  >  1$, $0 < s < 1$, $\alpha\in \mR$, $\Lambda > 1$, open $\Omega \subset \R^d$, and a measurable function $g$ defined in $\Omega$,  we set 
\begin{equation} \label{def-dot-Wsp}
\| g \|_{\dot W^{s, p, \alpha, \Lambda}(\Omega)}^p =\int_{\Omega} \int_{\Omega} \frac{|g(x) - g(y)|^p |x|^{\alpha p}}{|x - y|^{d + sp }} \chi_{\Lambda} (|x|, |y|) \, dx \, dy, 
\end{equation}
where, for $r_1, r_2  \ge 0$,  we denote
\begin{equation}\label{def-chi}
\chi_{\Lambda} (r_1, r_2) = \left\{  \begin{array}{cl} 1  &  \mbox{ for  } \Lambda^{-1} r_1 \le r_2 \le \Lambda r_1, \\[6pt]
0 &  \mbox{ otherwise}. 
\end{array}\right. 
\end{equation}
The dot in the LHS of \eqref{def-dot-Wsp} means that only the information of the ``semi-norm" is considered.

\medskip 
Our first main result is the following one dealing with the case where $1/ \tau + \gamma/ d \neq 0$.

\begin{thm}\label{thm1-rad}
Let $d\geq 2$, $0<s < 1$, $p>1$, $q\geq 1$, $\tau \ge 1$, $0 <  a\leq 1$, $\alpha, \, \beta, \,  \gamma \in \ro$, and $\Lambda > 1$. Define  $\sigma$ by \eqref{def-alpha-gamma}.  Assume \eqref{balance-law2} and
\begin{equation}\label{cond-s-1-*}
- (d-1) s  \leq \alpha-\sigma < 0. 
\end{equation}
We have, for some positive constant $C$, 
\begin{itemize}
\item[$i)$] if $\frac{1}{\tau}+\frac{\gamma}{d}>0$, then for all radial $u\in L^1_{loc}(\ro^d \setminus \{0 \})$ with compact support in $\mR^d$, it holds 
\begin{equation}\label{thm1-rad-1}
\nrm{|x|^\gamma u }_{L^\tau(\ro^d)}\leq C \nrm{u}^a_{\dot{W}^{s,p,\alpha, \Lambda}(\ro^d)}\nrm{|x|^\beta u}^{1-a}_{L^q(\ro^d)},
\end{equation}

\item[$ii)$] if $\frac{1}{\tau}+\frac{\gamma}{d}<0$, then for all radial $u\in L_{loc}^1(\ro^d \setminus \{0 \})$ which is 0 in a neighborhood of $0$, it holds 
\begin{equation}\label{thm1-rad-2}
\nrm{|x|^\gamma u }_{L^\tau(\ro^d)}\leq C \nrm{u}^a_{\dot{W}^{s,p,\alpha, \Lambda}(\ro^d)}\nrm{|x|^\beta u}^{1-a}_{L^q(\ro^d)}. 
\end{equation}
  
\end{itemize}
\end{thm}

\begin{remark} \label{rem-convention} \rm In \eqref{thm1-rad-1} and \eqref{thm1-rad-2}, the following convention is used: $+\infty . 0 = 0 . (+ \infty) = 0 $, $(+\infty)^0 = 1$ (this corresponds to the case $a=1$), and $+ \infty \le + \infty$. 
\end{remark}

\begin{remark}\label{rem-necessity}
The condition $\alpha-\sigma\geq -(d-1)s$ is in fact optimal, see \Cref{optimality}. 
\end{remark}

\begin{remark} \rm
Combing \eqref{CKN-fractional1}, \eqref{CKN-fractional2},  and \Cref{thm1-rad} yields that, in the radial case, \eqref{CKN-fractional1} and \eqref{CKN-fractional2} hold if one replaces \eqref{cond-s-1} and \eqref{cond-s-2} by the condition $- (d-1) s  \leq \alpha-\sigma$ and \eqref{cond-s-2}. 

\end{remark}

\medskip 
Concerning the limiting case $1/ \tau +  \gamma/ d = 0 $, we obtain the following result. 
 
\begin{thm}\label{thm2-rad}
Let $d\geq 2$, $0< s < 1$,  $p>1$, $q\geq 1$, $\tau \ge  1$, $0< a\leq 1$, $\alpha, \, \beta, \,  \gamma \in \ro$, $\mu >  1$, and $\Lambda > 1$. Assume that $\tau \le \mu$. Define $\sigma$ by \eqref{def-alpha-gamma}. Assume \eqref{balance-law2} and \eqref{cond-s-1-*}. 
There exists a positive constant $C$ such that for all  radial $u\in L^1_{loc}(\ro^d \setminus \{0 \})$ and for all $R_1, R_2 > 0$,  we have 
\begin{itemize}
\item[$i)$] if $\frac{1}{\tau}+\frac{\gamma}{d}=0$ and $\mathrm{supp}\ u \subset B_{R_2}$, then it  holds; 
\begin{equation}\label{thm2-rad-1}
\bct{\int_{\ro^d} \frac{|x|^{\gamma \tau}}{\ln^{\mu} (2R_2/|x|)}|u|^{\tau}dx}^\frac{1}{\tau}\leq C \nrm{u}^a_{\dot{W}^{s,p,\alpha, \Lambda}(\ro^d)}\nrm{|x|^\beta u}^{1-a}_{L^q(\ro^d)}, 
\end{equation}

 \item[$ii)$] if $\frac{1}{\tau}+\frac{\gamma}{d}=0$, and $\mathrm{supp} \ u \cap B_{R_1}= \emptyset$, then it holds
 \begin{equation}\label{thm2-rad-2}
\bct{\int_{\ro^d} \frac{|x|^{\gamma \tau}}{\ln^{\mu} (2|x|/ R_1)}|u|^{\tau}dx}^\frac{1}{\tau}\leq C \nrm{u}^a_{\dot{W}^{s,p,\alpha, \Lambda}(\ro^d)}\nrm{|x|^\beta u}^{1-a}_{L^q(\ro^d)}. 
\end{equation}
 
\end{itemize}
\end{thm}

\begin{remark}  \rm The convention in \Cref{rem-convention} is also used in \Cref{thm2-rad}.  
\end{remark}

\begin{remark} \label{log-term-necessity}\rm
If $1/q + \beta / d > 0$, by considering a smooth function $u$ which is 1 in a neighborhood of $0$ we can establish the necessity of the $\log$-term in $i)$ of \Cref{thm2-rad} . Similarly, if $1/q + \beta / d < 0$, by considering a smooth function $u$ which is 1 outside $B_R$ for some large $R$ the necessity of the $\log$-term in $ii)$ can be established. 
\end{remark}

\begin{remark} \rm
Combing \eqref{CKN-limiting-1}, \eqref{CKN-limiting-2},  and \Cref{thm2-rad} yields that, in the radial case, \eqref{CKN-limiting-1}, \eqref{CKN-limiting-2} hold if one replaces \eqref{cond-s-3}  by the condition $- (d-1) s  \leq \alpha-\sigma \le s$. 

\end{remark}


There are very few results known for the extended range of  the CKN inequalities in the fractional Sobolev spaces for radial functions (the case $s=1$ will be discussed in the last paragraph of \Cref{sect-s=1}). It was shown by Rubin \cite{Rubin83} (see also \cite[Theorem 4.3]{BGMMV}) that \eqref{CKN-fractional1} holds under the assumption  \eqref{cond-s-1-*} and  $1/ \tau + \gamma/ d > 0$ in  the case where $a=1$, $\tau \ge p = 2$, and $\alpha =0$. The same result was proved in \cite[Theorem 1.2]{DDD12}. These proofs are based on  inequalities for fractional integrations.  Our proof is different and quite elementary. It is  based on an improvement of the fractional CKN inequalities in one dimensional case and a simple use of polar coordinates. This strategy can be easily extended to other contexts. The  improvement was implicitly appeared in \cite{Ng-Squa1} and will be described briefly later. 
The same idea can be applied to the case $s=1$ and will be presented in \Cref{sect-s=1}. Applications to the compact embedding will be given in \Cref{sect-compact}. In particular, we derive the compact embedding of $W^{s, p}(\mR^d)$ into $L^q(\mR^d)$ 
for radial functions if $p  <  q < \frac{dp}{d- sp}$ for $0< s \le 1$ and $sp < d$. This result was previously obtained via various technique such as Strauss' lemma, Riesz-potential, fractional integration, Rubin's lemma, atomic decomposition, etc.

\medskip 
The paper is organized as follows. The improvement of the fractional CKN inequalities are given in \Cref{sect-improvement}. The proofs of \Cref{thm1-rad} and \Cref{thm2-rad} are given in \Cref{sect-rad}. The results in the case $s=1$ are given in \Cref{sect-s=1}. \Cref{sect-compact} is devoted to  the compactness results.

\section{Improvements of the fractional Caffarelli-Kohn-Nirenberg inequalities}\label{sect-improvement}

In this section, we will establish slightly more general versions of the fractional CKN inequalities. These improvements appear very naturally in the proof of  \Cref{thm1-rad} and \Cref{thm2-rad} when polar coordinates are used. 

\medskip 
We begin with an improvement of \eqref{CKN-fractional1} and \eqref{CKN-fractional2}.

\begin{thm}\label{thm1-CKNsw}
Let  $d\geq 1$, $0< s <  1$, $p>1$, $q\geq 1$, $\tau \ge 1$, $0 < a \leq 1$, $\alpha$, $\beta$, $\gamma \in \ro$, and $\Lambda > 1$. Define $\sigma$ by \eqref{def-alpha-gamma}.  Assume \eqref{balance-law2}, \eqref{cond-s-1}, and \eqref{cond-s-2}. There exists a positive constant $C$ such that
\begin{itemize}
\item[$i)$] if $\frac{1}{\tau}+\frac{\gamma}{d}>0$, then for all $u\in L^1_{loc}(\ro^d \setminus \{0\})$ with compact support in $\mR^d$,  it holds 
\begin{equation}
\nrm{|x|^\gamma u }_{L^\tau(\ro^d)}\leq C \nrm{u}^a_{\dot{W}^{s,p,\alpha, \Lambda}(\ro^d)}\nrm{|x|^\beta u}^{1-a}_{L^q(\ro^d)},
\end{equation}
 \item[$ii)$] if $\frac{1}{\tau}+\frac{\gamma}{d}<0$,  then for all $u\in L^1_{loc}(\ro^d \setminus \{0 \})$ which is 0 in a neighborhood of 0,  it holds 
\begin{equation}
\nrm{|x|^\gamma u }_{L^\tau(\ro^d)}\leq C \nrm{u}^a_{\dot{W}^{s,p,\alpha, \Lambda}(\ro^d)}\nrm{|x|^\beta u}^{1-a}_{L^q(\ro^d)}. 
\end{equation}
\end{itemize}
\end{thm}

Concerning an improvement of \eqref{CKN-limiting-1} and \eqref{CKN-limiting-2}, we have  the following result. 

\begin{thm}\label{thm2-CKNsw}
Let $d\geq 1$, $0< s <  1$, $p>1$, $q\geq 1$, $\tau \ge 1$, $0 < a \leq 1$, $\alpha$, $\beta$, $\gamma \in \ro$, $\mu > 1$, and $\Lambda > 1$. Assume that $\tau \le \mu$.  Define $\sigma $ by \eqref{def-alpha-gamma}.    Assume \eqref{balance-law2} and  \eqref{cond-s-3}.  There exists a positive constant $C$ such that for all $u\in L^1_{loc}(\ro^d \setminus \{0 \})$ and for all $R_1, R_2 > 0$,  we have 
\begin{itemize}
\item[(i)] if $\frac{1}{\tau}+\frac{\gamma}{d}=0$ and $\mathrm{supp}\ u \subset B_{R_2}$, then it holds 
\begin{equation}
\bct{\int_{\ro^d} \frac{|x|^{\gamma \tau}}{\ln^{\mu} (2R_2/|x|)}|u|^{\tau}dx}^\frac{1}{\tau}\leq C \nrm{u}^a_{\dot{W}^{s,p,\alpha, \Lambda}(\ro^d)}\nrm{|x|^\beta u}^{1-a}_{L^q(\ro^d)}, 
\end{equation}
 \item[(ii)] if $\frac{1}{\tau}+\frac{\gamma}{d}=0$, and $\mathrm{supp} \ u \cap B_{R_1}= \emptyset$, then it holds 
 \begin{equation}
\bct{\int_{\ro^d} \frac{|x|^{\gamma \tau}}{\ln^{\mu} (2|x|/R_1)}|u|^{\tau}dx}^\frac{1}{\tau}\leq C \nrm{u}^a_{\dot{W}^{s,p,\alpha, \Lambda}(\ro^d)}\nrm{|x|^\beta u}^{1-a}_{L^q(\ro^d)}. 
\end{equation}
\end{itemize}
\end{thm}

It is clear that \Cref{thm1-CKNsw} implies \eqref{CKN-fractional1} and \eqref{CKN-fractional2} and \Cref{thm2-CKNsw} yields \eqref{CKN-limiting-1}  and \eqref{CKN-limiting-2}.  \Cref{thm1-CKNsw} and \Cref{thm2-CKNsw} were already implicitly contained in \cite{Ng-Squa1} where \eqref{CKN-fractional1}, \eqref{CKN-fractional2}, \eqref{CKN-limiting-1}, and \eqref{CKN-limiting-2} were established.   
For the convenience of the reader, we will describe briefly the proofs of \Cref{thm1-CKNsw} and \Cref{thm2-CKNsw} in the next two sections respectively.

\subsection{Proof of \Cref{thm1-CKNsw}} The proof is divided into two steps where we prove $i)$ and $ii)$ respectively.

\medskip 
\noindent {\bf Step 1: Proof of $i)$.}  For simplicity of arguments,  we assume that $\Lambda > 4$ from later on \footnote{In the general case, one just needs to define $\mcal{A}_k$ by $\Sb{x\in \ro^d; \; \lambda^k\leq |x| < \lambda^ {k+1}}$ with $\lambda^2 = \Lambda$ instead of \eqref{def-Ak}.}. 

We first consider the case $0 \le \alpha-\sigma \leq s$.  As in \cite{Ng-Squa1}, for $k\in \mathbb{Z}$ set 
\begin{equation}\label{def-Ak}
\mcal{A}_k : = \Sb{x\in \ro^d; \; 2^k\leq |x| <2^ {k+1}}.
\end{equation}
Since $\alpha-\sigma \ge 0$, by Gagliardo-Nirenberg inequality  \cite[Lemma 2.2]{Ng-Squa1} \footnote{\cite[Lemma 2.2]{Ng-Squa1}  states for functions of class $C^1$ up to the boundary but the same result holds for our setting by using the standard convolution technique.}, we derive  that 
\begin{multline}\label{thm1-p1}
\bct{\fint_{\mcal{A}_k} \left|u- \fint_{\mcal{A}_k} u \right|^{\tau}}^{\frac{1}{\tau}} \\[6pt]
 \leq   C \bct{2^{-(d-sp)k}\int_{\mcal{A}_k}\int_{\mcal{A}_k} \frac{|u(x)-u(y)|^p}{|x-y|^{d+sp}}dxdy}^{a/p} \bct{\fint_{\mcal{A}_k}|u(x)|^qdx}^{(1-a)/q}.
\end{multline}
Here and in what follows in the proof of \Cref{thm1-CKNsw}, $C$ denotes a positive constant independent of $u$ and $k$ (and also independent of $m$, and $n$, which appear later), and $\fint_{\Omega}: = \frac{1}{|\Omega|} \int_{\Omega}$.  Since 
$$
2^{ \tau\gamma k} \int_{\mcal{A}_k} |u|^{\tau}  \le C  2^{(\tau \gamma+d)k} \fint_{\mcal{A}_k} \left|u- \fint_{\mcal{A}_k} u\right|^{\tau} + C 2^{(\tau \gamma+d)k}\md{\fint_{\mcal{A}_k}u}^{\tau}, 
$$
using  \eqref{balance-law2}, we derive from \eqref{thm1-p1} that 
 \begin{multline}\label{sameAnnEst}
 \int_{\mcal{A}_k} |u|^{\tau} |x|^{\tau \gamma} \, dx   \leq  C 2^{(\gamma\tau+d)k}\md{\fint_{\mcal{A}_k}u}^{\tau} \\[6pt] + C \bct{\int_{\mcal{A}_k}\int_{\mcal{A}_k} \frac{|u(x)-u(y)|^p |x|^{\alpha p}}{|x-y|^{d+sp}} \, dx \, dy}^\frac{a\tau}{p}   \bct{\int_{\mcal{A}_k} |u(x)|^q |x|^{\beta q}\, dx }^\frac{(1-a)\tau}{q}. 
\end{multline}

Let $m, \, n\in \mathbb{Z}$ be such that $m\leq n-2$ and $\mathrm{supp} \ u \subset B_{2^n}$. Summing \eqref{sameAnnEst} with respect to $k$ from $m$ to $n$,  we get
\begin{multline}\label{thm1-CKNsw-m1}
\int_{\Sb{2^m<|x|<2^{n+1}}} |u|^{\tau} |x|^{\tau\gamma}   \leq  C\sum_{k=m}^n 2^{(\gamma\tau+d)k}\md{\fint_{\mcal{A}_k}u}^{\tau} \\[6pt]
+ C\sum_{k=m}^n \bct{\int_{\mcal{A}_k}\int_{\mcal{A}_k} \frac{|u(x)-u(y)|^p |x|^{\alpha p} }{|x-y|^{d+sp}}\, dx \, dy}^\frac{a\tau}{p} \nrm{|x|^\beta u}^{(1-a)\tau}_{L^q(\mcal{A}_k)}. 
\end{multline}
Applying \Cref{lem-Ineq} below with $\kappa = a \tau / p$ and $\eta = (1-a) \tau /q$ after using the condition $\alpha-\sigma\leq s$ to check that $\kappa + \eta \ge 1$, we derive that
\begin{equation}\label{thm1-CKNsw-m2}
\sum_{k=m}^n \bct{\int_{\mcal{A}_k}\int_{\mcal{A}_k} \frac{|u(x)-u(y)|^p |x|^{\alpha p} }{|x-y|^{d+sp}}\, dx \, dy}^\frac{a\tau}{p} \nrm{|x|^\beta u}^{(1-a)\tau}_{L^q(\mcal{A}_k)} \le \| u\|_{\dot W^{s, p, \alpha, \Lambda} (\mR^d)}^{a\tau} \nrm{|x|^\beta u}^{(1-a)\tau}_{L^q(\ro^d)}. 
\end{equation}
Combining \eqref{thm1-CKNsw-m1} and \eqref{thm1-CKNsw-m2} yields
\begin{equation}\label{thm1-p2}
\int_{\Sb{|x|>2^m}} |u|^{\tau} |x|^{\tau\gamma}     \leq  C\sum_{k=m}^n 2^{(\gamma\tau+d)k}\md{\fint_{\mcal{A}_k}u}^{\tau} + C \| u\|_{\dot W^{s, p, \alpha, \Lambda} (\mR^d)}^{a\tau} \nrm{|x|^\beta u}^{(1-a)\tau}_{L^q(\ro^d)}.  \end{equation}

We next estimate the first term of the RHS of \eqref{thm1-p2}. We have, as in \eqref{thm1-p1},
\begin{multline}\label{thm1-p3}
\md{\fint_{\mcal{A}_k} u-\fint_{\mcal{A}_{k+1}}u}^{\tau} \leq   C  \bct{2^{(d-sp)k}\int_{\mcal{A}_k\cup \mcal{A}_{k+1}}\int_{\mcal{A}_k\cup\mcal{A}_{k+1}} \frac{|u(x)-u(y)|^p}{|x-y|^{d+sp}} \, dx \, dy}^\frac{a\tau}{p} \\[6pt] \times \bct{\fint_{\mcal{A}_k\cup \mcal{A}_{k+1}} |u(x)|^q \, dx }^\frac{(1-a)\tau}{q}. 
\end{multline}
With $c =2/(1+2^{\gamma\tau+d})<1$, since $c 2^{\gamma\tau+d} > 1$  thanks to $\gamma \tau + d > 0$ we derive from \eqref{thm1-p3} that 
\begin{multline}\label{thm1-coucou}
2^{(\gamma\tau+d)k} \md{\fint_{\mcal{A}_k} u}^{\tau} \leq c 2^{(\gamma\tau+d)(k+1)} \md{\fint_{\mcal{A}_{k+1}}u} ^{\tau} \\[6pt]
+C \bct{\int_{\mcal{A}_k\cup\mcal{A}_{k+1}}\int_{\mcal{A}_k\cup\mcal{A}_{k+1}} \frac{|u(x)-u(y)|^p |x|^{\alpha p}}{|x-y|^{d+sp}}dxdy}^\frac{a\tau}{p}   \nrm{|x|^\beta u}^{(1-a)\tau}_{L^q(\mcal{A}_k\cup \mcal{A}_{k+1})}.
\end{multline}
Summing this inequality with respect to $k$ from $m$ to $n$ for large $n$, since $u$ has a compact support in $B_{2^n}$ and $c < 1$  thanks to $\gamma \tau + d > 0$, we derive  that  
\begin{multline}\label{sumAvgEst}
\sum_{k=m}^n 2^{(\gamma\tau+d)k} \md{\fint_{\mcal{A}_k} u}^{\tau} \leq C \sum_{k=m}^n\bct{\int_{\mcal{A}_k\cup\mcal{A}_{k+1}}\int_{\ro^d} \frac{|u(x)-u(y)|^p |x|^{\alpha p}}{|x-y|^{d+sp}} \, dx \, dy}^\frac{a\tau}{p}\\[6pt] \times \nrm{|x|^\beta u}^{(1-a)\tau}_{L^q(\mcal{A}_k\cup \mcal{A}_{k+1})}. 
\end{multline}
Applying \Cref{lem-Ineq} below again and letting $m \to - \infty$ , we obtain 
\begin{equation}\label{thm1-p4}
\sum_{k \in \mathbb{Z}} 2^{(\gamma\tau+d)k} \md{\fint_{\mcal{A}_k} u}^{\tau} \leq 
C \nrm{u}_{\dot W^{s, p, \alpha, \Lambda} (\mR^d)}^{a\tau}\nrm{|x|^\beta u}^{(1-a)\tau}_{L^q(\ro^d)} .  
\end{equation}

Combining \eqref{thm1-p2} and \eqref{thm1-p4} and letting $m \to - \infty$, we obtain $(i)$ of \Cref{thm1-CKNsw}. The proof of $i)$ in the case $0 \le \alpha - \sigma \le s$ is complete. 

The proof of $i)$ in the case $\alpha - \sigma > s$ and $\frac{1}{\tau}+\frac{\gamma}{d} \neq  \frac{1}{p}+\frac{\alpha-s}{d}$ is based on the standard interpolation technique as in \cite{CKN, Ng-Squa1}. One just notes that, for $\lambda > 0$,  
\begin{equation*}
\| u (\lambda \cdot) \|_{\dot W^{s, p, \alpha, \Lambda}(\mR^d)} = \lambda^{s-\alpha - \frac{d}{p}}  \| u  \|_{\dot W^{s, p, \alpha, \Lambda}(\mR^d)} 
\end{equation*}
since $\chi_{\Lambda} (x, y) = \chi_{\Lambda} (\lambda x, \lambda y)$, and 
\begin{equation*}
\| |x|^\gamma u (\lambda \cdot) \|_{L^\tau(\mR^d)} = \lambda^{- \gamma - \frac{d}{\tau}} \| |x|^\gamma u \|_{L^\tau(\mR^d)} \quad \mbox{ and } \quad \| |x|^\beta u (\lambda \cdot) \|_{L^q (\mR^d)} =  \lambda^{- \beta - \frac{d}{q}} \| |x|^\beta u  \|_{L^q (\mR^d)}. 
\end{equation*}
The details are omitted.

\medskip 
\noindent {\bf Step 2: Proof of $ii)$.} The proof of $ii)$ of \Cref{thm1-CKNsw} is similar to that of $i)$.  We only deal with the case $0 \le \alpha - \sigma \le s$ since the proof in the case where $\alpha - \sigma > s$ and $\frac{1}{\tau}+\frac{\gamma}{d} \neq  \frac{1}{p}+\frac{\alpha-s}{d}$ is only by interpolation and almost unchanged.

Assume $0 \le \alpha - \sigma \le s$. Let $m$ be such that $u =0$ in $B_{2^m}$.  Similar to  \eqref{thm1-p2}, we have
\begin{equation}\label{thm1-p2-*}
\int_{\Sb{|x|< 2^n}} |u|^{\tau} |x|^{\tau\gamma}     \leq  C\sum_{k=m}^n 2^{(\gamma\tau+d)k}\md{\fint_{\mcal{A}_k}u}^{\tau} + C \| u\|_{\dot W^{s, p, \alpha, \Lambda} (\mR^d)}^{a\tau} \nrm{|x|^\beta u}^{(1-a)\tau}_{L^q(\ro^d)}.  
\end{equation}

 To estimate the first term in RHS of \eqref{thm1-p2-*}, one just needs to note that, instead of \eqref{thm1-coucou}, we have 
with $c=(1+2^{\gamma\tau+d})/2<1$  thanks to  $\gamma \tau  + d <  0$, 
\begin{multline*}
2^{(\gamma\tau+d)(k+1)} \md{\fint_{\mcal{A}_{k+1}} u}^{\tau}   \leq c 2^{(\gamma\tau+d)k} \md{\fint_{\mcal{A}_{k}}u} ^{\tau} \\[6pt]
+C \bct{\int_{\mcal{A}_k\cup\mcal{A}_{k+1}}\int_{\mcal{A}_k\cup\mcal{A}_{k+1}} \frac{|u(x)-u(y)|^p |x|^{\alpha p} }{|x-y|^{d+sp}}dxdy}^\frac{a\tau}{p}   \nrm{|x|^\beta u}^{(1-a)\tau}_{L^q(\mcal{A}_k\cup \mcal{A}_{k+1})}. \end{multline*}
Summing with respect to $k$, we also obtain \eqref{thm1-p4}. The conclusion now follows from  \eqref{thm1-p4} and \eqref{thm1-p2-*}. \qed

\medskip 

The following simple lemma is used in the proof of \Cref{thm1-CKNsw}. 

\begin{lem}\label{lem-Ineq} For $\kappa, \eta \ge 0$ with $\kappa + \eta \ge 1$, and $k \in \mN$, we have 
\begin{equation}
\sum_{i=1}^k   |a_{i}|^{\kappa} |b_{i}|^{\eta}    \leq  \bct{\sum_{i=1}^k |a_{i}|}^{\kappa} \bct{\sum_{i=1}^k |b_{i}|}^{\eta}  \mbox{ for } a_{i},  \, b_i   \in \mR.
\end{equation}
\end{lem}

\subsection{Proof of \Cref{thm2-CKNsw}}

As in the proof of \Cref{thm1-CKNsw}, we assume that $\Lambda > 4$ for notational ease.  In this proof, we use the notations in the proof of Theorem~\ref{thm1-CKNsw}.  We only prove the first assertion. The second assertion follows similarly as in the spirit of the proof of Theorem~\ref{thm1-CKNsw}.  
Let $n\in \mathbb{N}$ be such that $ 2^{n-1} \le R_2 < 2^{n}$.

Set 
\begin{equation}
\nu = \mu - 1 > 0. 
\end{equation}
Since $\alpha - \sigma \ge 0$, using \eqref{balance-law2},  we also obtain \eqref{sameAnnEst}.  Summing \eqref{sameAnnEst} with respect to $k$ from $m$ to $n$, we obtain 
\begin{multline}\label{CKN-part1-2}
\int_{\{ |x| > 2^{m} \}} \frac{1}{\ln^{1 + \nu} (2 R_2/ |x|)} |x|^{\gamma \rr}|u|^\rr \, dx  \\[6pt]
\le C \sum_{k = m}^n \frac{1}{(n-k+1)^{1 + \nu}} \Big|\fint_{\C_k} u\Big|^\rr   + C \sum_{k = m}^n \bct{\int_{\mcal{A}_k}\int_{\mcal{A}_k} \frac{|u(x)-u(y)|^p |x|^{\alpha p} }{|x-y|^{d+sp}}\, dx \, dy}^\frac{a\tau}{p}
 \| |x|^\beta u   \|_{L^q(\C_k)}^{(1-a)\rr}. 
\end{multline}

As in \eqref{thm1-p3}, we have 
\begin{multline}\label{CKN-part1-2-1}
\left| \fint_{\C_{k}} u - \fint_{\C_{k+1}} u \right|^\tau \le C  \bct{2^{(d-sp)k}\int_{\mcal{A}_k\cup \mcal{A}_{k+1}}\int_{\mcal{A}_k\cup\mcal{A}_{k+1}} \frac{|u(x)-u(y)|^p}{|x-y|^{d+sp}} \, dx \, dy}^\frac{a\tau}{p} \\[6pt] 
\times \bct{\fint_{\mcal{A}_k\cup \mcal{A}_{k+1}} |u(x)|^q \, dx }^\frac{(1-a)\tau}{q} .
\end{multline}

Applying \Cref{lem-Holder} below with $c = (n-k+1)^\nu/ (n-k+1/2)^\nu$, we deduce  that
\begin{equation*}
 \left| \fint_{\C_{k}} u \right|^{\rr} \le \frac{(n - k + 1)^{\nu}}{(n - k + 1/2)^{\nu}}  \left| \fint_{\C_{k+1}} u \right|^{\rr} + 
 C (n - k + 1)^{\rr - 1} \left| \fint_{\C_{k}} u - \fint_{\C_{k+1}} u \right|^\tau, 
\end{equation*}
since, for $\nu > 0$,  
$$
(n-k+1)^\nu/ (n-k+1/2)^\nu -1 \sim \frac{1}{n-k +1}. 
$$
It follows from  \eqref{balance-law2} and \eqref{CKN-part1-2-1} that 
\begin{multline*}
 \left| \fint_{\C_{k}} u \right|^{\rr} \le \frac{(n - k + 1)^{\nu}}{(n - k + 1/2)^{\nu}}  \left| \fint_{\C_{k+1}} u \right|^{\rr}  \\[6pt]
 +C (n - k + 1)^{\rr - 1} \bct{\int_{\mcal{A}_k\cup\mcal{A}_{k+1}}\int_{\mcal{A}_k\cup\mcal{A}_{k+1}} \frac{|u(x)-u(y)|^p |x|^{\alpha p} }{|x-y|^{d+sp}}dxdy}^\frac{a\tau}{p}  \nrm{|x|^\beta u}^{(1-a)\tau}_{L^q(\mcal{A}_k\cup \mcal{A}_{k+1})}.  
\end{multline*}

\pagebreak
 
This yields 
\begin{multline}\label{CKN-part2-2}
\frac{1}{(n - k + 1)^{\nu}} \left| \fint_{\C_{k}} u \right|^{\rr} \le \frac{1}{(n - k + 1/2)^{\nu}}  \left| \fint_{\C_{k+1}} u \right|^{\rr} \\[6pt]+   C(n - k + 1)^{\rr -1 - \nu}  \bct{\int_{\mcal{A}_k\cup\mcal{A}_{k+1}}\int_{\mcal{A}_k\cup\mcal{A}_{k+1}} \frac{|u(x)-u(y)|^p |x|^{\alpha p} }{|x-y|^{d+sp}}dxdy}^\frac{a\tau}{p}   \nrm{|x|^\beta u}^{(1-a)\tau}_{L^q(\mcal{A}_k\cup \mcal{A}_{k+1})}.
\end{multline}

We have, for $\nu > 0$ and $k \le n$,  
\begin{equation}\label{CKN-part3-2}
\frac{1}{(n - k + 1)^{\nu}} - \frac{1}{(n - k + 3/2)^{\nu}} \sim \frac{1}{(n - k + 1)^{\nu + 1}}
\end{equation}
and, since $\tau \le 1 + \nu$,  
\begin{equation}\label{CKN-part3-3}
(n - k + 1)^{\rr -1 - \nu}  \le 1. 
\end{equation}
Summing \eqref{CKN-part2-2} from $m$ to $n$,  and using  \eqref{CKN-part3-2} and \eqref{CKN-part3-3}, we derive  that
\begin{multline}\label{CKN-part4-2}
\sum_{k = m}^n  \frac{1}{(n-k+1)^{1 + \nu}}  \Big| \fint_{\C_{k}} u \Big|^\rr \\[6pt] \le C \sum_{k = m}^n  \bct{\int_{\mcal{A}_k\cup\mcal{A}_{k+1}}\int_{\mcal{A}_k\cup\mcal{A}_{k+1}} \frac{|u(x)-u(y)|^p |x|^{\alpha p} }{|x-y|^{d+sp}}dxdy}^\frac{a\tau}{p}   \nrm{|x|^\beta u}^{(1-a)\tau}_{L^q(\mcal{A}_k\cup \mcal{A}_{k+1})}.
\end{multline}

Combining \eqref{CKN-part1-2} and \eqref{CKN-part4-2},  we obtain  
\begin{multline*}
\int_{\{ |x|  > 2^{m} \}} \frac{|x|^{\gamma \rr}}{\ln^{1 + \nu} (2^{n+1}/|x|)}  |u|^\rr \, dx  \\[6pt] \le C\sum_{k = m}^n \bct{\int_{\mcal{A}_k\cup\mcal{A}_{k+1}}\int_{\mcal{A}_k\cup\mcal{A}_{k+1}} \frac{|u(x)-u(y)|^p |x|^{\alpha p} }{|x-y|^{d+sp}}dxdy}^\frac{a\tau}{p}   \nrm{|x|^\beta u}^{(1-a)\tau}_{L^q(\mcal{A}_k\cup \mcal{A}_{k+1})}. 
\end{multline*}

Applying \Cref{lem-Ineq} with $\kappa =  a \rr /p$ and $\eta = (1-a) \rr / q$,  we derive that
\begin{equation*}
\int_{\{ |x|  > 2^{m} \}} \frac{|x|^{\gamma \rr}}{\ln^{1 + \nu} (2^{n+1}/|x|)}  |u|^\rr \, dx   \le C  \|u\|_{\dot{W}^{s, p , \alpha, \Lambda}(\ro^d)}^{a \rr} \| |x|^\beta u   \|_{L^q( \bigcup_{k = m}^\infty \C_k )}^{(1-a)\rr}.
\end{equation*}
This yields the conclusion. \qed

\medskip 
In the proof of \Cref{thm2-CKNsw}, we used the following elementary lemma which was stated in 
 \cite[Lemma 3.2]{Ng-Squa1}. For the completeness, we give the proof below. 
 
\begin{lem}\label{lem-Holder} Let  $M> 1$ and $\tau \ge 1$.  There exists $C = C(M, \tau) > 0$, depending only on $M$ and $\tau$ such that, for all $1 < c < M$, 
\begin{equation}
\label{lem-Holder-st1}
(|a| + |b|)^\rr \le c |a|^\rr + \frac{C}{(c - 1)^{\rr -1}} |b|^\rr \mbox{ for all } a, b \in \R.  
\end{equation}
\end{lem}

\begin{proof} The inequality is trivial when $\tau =1$. We next only deal with the case $\tau > 1$. 

Without loss of generality, one might assume that $a \ge 0$ and $b \ge 0$. Inequality \eqref{lem-Holder-st1} is clear if $a = 0$ or $b =0$. Thus it suffices to consider the case where $a > 0$ and $b > 0$. This will be assumed from now on. 
Set $x = a/b$. Multiplying two sides of the inequality by $a^{-\tau}$, it is enough to prove that, for some $C>0$, 
\begin{equation}\label{lem-Holder-p1}
(1 + x)^{\tau} \le c  +  \frac{C}{(c - 1)^{\rr -1}}  x^{\tau} \mbox{ for } x > 0. 
\end{equation}
There exists $x_0>0$ such that, for $0< x < x_0$,  
$$
(1 + x)^{\tau} \le 1 + 2 \tau x.
$$
On the other hand, we have 
$$
c  +  \frac{C}{(c - 1)^{\rr -1}}  x^{\tau} = 1 + (c-1) + \frac{C}{(c - 1)^{\rr -1}}  x^{\tau} \ge 1 + \frac{\tau-1}{\tau}(c-1) + \frac{1}{\tau}\frac{C}{(c - 1)^{\rr -1}}  x^{\tau}. 
$$
Applying the Young inequality, we obtain 
$$
 \frac{\tau-1}{\tau}(c-1) + \frac{1}{\tau}\frac{C}{(c - 1)^{\rr -1}}  x^{\tau} \ge (c-1)^{\frac{\tau-1}{\tau}} \frac{C^{\frac{1}{\tau}}x }{(c - 1)^{\frac{\rr -1}{\tau}}}  \ge 2 x \quad \mbox{ if } \quad C  > C_1  : = 2^{\tau}. 
$$
Thus \eqref{lem-Holder-p1} holds for $0 < x < x_0$ for $C \ge C_1$. 

It is clear that there exists $C_2>0$ such that \eqref{lem-Holder-p1} holds for $x \ge x_0$ for $C \ge C_2$. 

By choosing $C = \max\{C_1, C_2\}$, we obtain \eqref{lem-Holder-p1} and the conclusion follows. 
\end{proof}

\begin{remark} \rm \Cref{lem-Holder} is stated in \cite{Ng-Squa1} for $\tau > 1$. Nevertheless, the result is trivial for $\tau =1$. 
\end{remark}

\section{The Caffarelli-Kohn-Nirenberg inequalities for radial functions in the fractional Sobolev spaces}\label{sect-rad}

This section containing two subsections is devoted to the proofs of \Cref{thm1-rad} and \Cref{thm2-rad}. In the first subsection, we present a lemma which  brings the situation in the radial case into the one of one dimensional space via polar coordinates. The proof of \Cref{thm1-rad} is given in the second subsection by applying \Cref{thm1-CKNsw} in one dimensional space and using the lemma in the first subsection.

\subsection{A useful lemma}

The improvement forms of the CKN inequalities are  inspired by the following lemma. 

\begin{lem}\label{lem-rad1}
Let $d \geq 2$, $0<s<1$, $1 \leq p<\infty$, $\alpha \in \ro$, $\Lambda > 1$,  and let  $u \in L^1_{loc}(\ro^d \setminus \{0 \})$ be  radial.  
We have, with $\hu(r) = u(r \sigma)$ for some $\sigma \in \mS^{d-1}$ and  for $r > 0$,  
\begin{multline}\label{lem-rad1-cl}
\int_{0}^\infty \int_{0}^\infty \frac{|\hu(r_1)- \hu(r_2)|^p r_1^{ \alpha p+ (d-1)} \chi_{\Lambda} (r_1, r_2)}{|r_1 - r_2|^{1+sp}} 
d r_1 d r_2 \\[6pt]
\leq C \int_{\ro^d} \int_{\ro^d} \frac{|u(x)-u(y)|^p |x|^{\alpha p} \chi_\Lambda (|x|, |y|)}{|x-y|^{d+sp}} \,  dx \,  dy,   
\end{multline}
where  $C$ is a positive constant depending only on $d, \, s$, $\alpha$, $p$, and $\Lambda$.
\end{lem}

\begin{proof} The proof is simply based on the use of the polar coordinates. Using these coordinates, we have  
\begin{multline}\label{lem-rad1-p1}
\int_{\ro^d}\int_{\ro^d} \frac{|u(x_1)-u(x_2)|^p|x|^{\alpha p} \chi_{\Lambda} (|x_1|, |x_2|)}{|x_1 - x_2|^{d+sp} } \, dx_1 \, d x_2 \\[6pt]
= \int_0^\infty \int_0^\infty |\hu(r_1)- \hu(r_2)|^p r_1^{\alpha p  + (d-1)} r_2^{d-1} \chi_{\Lambda} (r_1, r_2)  \int_{\mathbb S^{d-1}} \int_{\mathbb S^{d-1}} \frac{d \sigma_1 d\sigma_2}{|r_1 \sigma_1 - r_2 \sigma_2|^{d+sp}} \, dr_1 \, d r_2. 
\end{multline}
Since 
$$
|r_1 \sigma_1 - r_2 \sigma_2| = |(r_1- r_2) \sigma_1 + r_2 (\sigma_1 - \sigma_2)| \le  |r_1 - r_2| + |r_2|  |\sigma_1 - \sigma_2|, 
$$
it follows that, for $\Lambda r_1 \le r_2 \le \Lambda r_1$,
\begin{equation}\label{lem-rad1-p2}
\int_{\mathbb S^{d-1}} \int_{\mathbb S^{d-1}} \frac{d \sigma_1 d\sigma_2}{|r_1 \sigma_1 - r_2 \sigma_2|^{d+sp}} \ge C \int_{0}^1 \frac{s^{d-2} ds}{\big(|r_1 - r_2| + |r_2|  s\big)^{d+ sp}} \ge \frac{C}{r_2^{d-1} |r_1 - r_2|^{1 + sp}}. 
\end{equation} 
The conclusion now follows from \eqref{lem-rad1-p1} and \eqref{lem-rad1-p2}. 
\end{proof}

\subsection{Proof of \Cref{thm1-rad}}  Denote $\hu(r) = u(r \sigma)$ with $r > 0$ and $\sigma \in  \mS^{d-1}$. We have, by polar coordinates,  
\begin{equation}
\| |x|^\gamma u\|_{L^\tau(\mR^d)} = |\mS^{d-1}|^{\frac{1}{\tau}} \| r^{\gamma + \frac{d-1}{\tau}} \hu\|_{L^\tau(0, \infty)},  
\end{equation}
\begin{equation}\label{thm1-rad-m1}
\| |x|^\beta u\|_{L^q(\mR^d)} = |\mS^{d-1}|^{\frac{1}{q}} \| r^{\beta + \frac{d-1}{q}} \hu\|_{L^q(0, \infty)},
\end{equation}
and by \Cref{lem-rad1}, 
\begin{multline}\label{thm1-rad-m2}
\int_0^\infty \int_0^\infty \frac{|\hu(r_1) - \hu(r_2)|^p r_1^{\alpha p + d -1 } \chi_{\Lambda} (r_1, r_2)}{|r_1 - r_2|^{1 + sp }}  \, d r_1 \, dr_2  \\[6pt]
\le C \int_{\ro^d}\int_{\ro^d} \frac{|u(x)-u(y)|^p|x|^{\alpha p} \chi_\Lambda (|x|, |y|)}{|x-y|^{d+sp}}dx dy. 
\end{multline}

Extend $\hu$ in $\mR$ as an even function and still denote the extension by $\hu$. We have 
\begin{equation}
\| |\xi|^{\gamma + \frac{d-1}{\tau}} \hu\|_{L^\tau(\mR)}  \sim  \| r^{\gamma + \frac{d-1}{\tau}} \hu\|_{L^\tau(0, \infty)} 
\end{equation}
\begin{equation}
\| |\xi|^{\beta + \frac{d-1}{q}} \hu\|_{L^\tau(\mR)}  \sim \| r^{\beta + \frac{d-1}{q}} \hu\|_{L^q(0, \infty)},
\end{equation}
and 
\begin{multline}
\int_\mR \int_\mR \frac{|\hu(\xi_1) - \hu(\xi_2)|^p |\xi_1|^{\alpha p + d-1} \chi_{\Lambda} (|\xi_1|, |\xi_2|) }{|\xi_1 - \xi_2|^{1 + sp }} 
\, d \xi_1 \, d \xi_2  \\[6pt]
\le 4 \int_0^\infty \int_0^\infty \frac{|\hu(r_1) - \hu(r_2)|^p r_1^{\alpha p + d-1} \chi_{\Lambda} (r_1, r_2)}{|r_1 - r_2|^{1 + sp }}  \, d r_1 \, dr_2.  
\end{multline}
Hereafter in this proof, two quantities are $\sim$ if each one is bounded by the other up to a positive constant depending only on the parameters.

It thus suffices to prove 
\begin{equation}\label{thm1-rad-m3}
\| |\xi|^{\gamma + \frac{d-1}{\tau}} \hu\|_{L^\tau(\mR)} \le C \| \hu \|^a_{\dot{W}^{s, p, \alpha + \frac{d-1}{p}, \Lambda}(\mR)}\| |\xi|^{\beta + \frac{d-1}{q}} \hu\|_{L^\tau(\mR)}^{1-a}. 
\end{equation}
This is in fact a consequence of \Cref{thm1-CKNsw} in one dimensional case. To this end, 
let first rewrite the conclusion of \Cref{thm1-CKNsw} in one dimensional case. Let $0<s'<1$, $p'>1$, $q'\geq1$ $\tau'\geq1$, $0 <  a' \leq 1$, $\alpha'$, $\beta'$, $\gamma' \in \ro$ and define $\sigma'$ by  $\sigma' \in \ro$ by $\gamma'= a' \sigma'+ (1-a')\beta'$.  Assume that
\begin{equation}\label{thm1-rad-p1}
\frac{1}{\tau'}+\gamma'= a'\bct{\frac{1}{p'}+ \alpha'-s'}+(1-a')\bct{\frac{1}{q'}+\beta'}, 
\end{equation}
\begin{equation}\label{thm1-rad-p2}
0\leq \alpha'-\sigma', 
\end{equation}
and 
\begin{equation}\label{thm1-rad-p3}
\alpha'-\sigma' \leq s'  \text{ if } \frac{1}{\tau'}+ \gamma'= \frac{1}{p'}+\alpha'-s'.
\end{equation}
Then, if $\frac{1}{\tau'}+\gamma'>0$,  it holds 
\begin{equation}\label{thm1-rad-p4}
\nrm{|x|^{\gamma'} g }_{L^{\tau'}(\ro)}\leq C \nrm{g}^{a'}_{\dot{W}^{s',p',\alpha', 4}(\ro)}\nrm{|x|^{\beta'} g}^{1-a'}_{L^{q'}(\ro)} \mbox{ for } g \in L^1_{loc}(\mR\setminus \{0\}), \text{ with compact support in } \ro, 
\end{equation}
and if $\frac{1}{\tau'}+\gamma' < 0$,  it holds 
\begin{equation}\label{thm1-rad-p5}
\nrm{|x|^{\gamma'} g }_{L^{\tau'}(\ro)}\leq C \nrm{g}^{a'}_{\dot{W}^{s',p',\alpha', 4}(\ro)}\nrm{|x|^{\beta'} g}^{1-a'}_{L^{q'}(\ro)} \mbox{ for } g \in L^1_{loc}(\mR) \mbox{ with } 0 \not \in \mbox{supp } g.
\end{equation}

We are applying \eqref{thm1-rad-p4} and \eqref{thm1-rad-p5} with $s'=s$, $a'=a$, $p'=p$, $q'=q$, $\tau'=\tau$, $\alpha'= \alpha + \frac{d-1}{p}$, $\beta'= \beta +  \frac{d-1}{q}$,  $\gamma'= \gamma + \frac{d-1}{\tau}$, $a\sigma'+(1-a)\beta'= \gamma'$. Then clearly,
$$\frac{1}{\tau'}+\gamma'=\frac{d}{\tau}+\gamma, \ \ \frac{1}{p'}+\alpha'-s'= \frac{d}{p}+\alpha-s, \ \ \frac{1}{q'}+\beta'=\frac{d}{q}+\beta.$$ 
Hence \eqref{thm1-rad-p1} follows from \eqref{balance-law2}. 

We next compute $\alpha'-\sigma'$. Since $a \sigma' + (1-a) \beta' = \gamma' = \gamma + \frac{d-1}{\tau}$ and  $a \sigma + (1-a) \beta = \gamma$, it follows that 
\begin{multline*}
a(\sigma' - \sigma) = \frac{d-1}{\tau} - (1 -a) (\beta' - \beta) =  \frac{d-1}{\tau} - \frac{(1-a) (d-1) }{q} \\[6pt] 
= (d-1) \left(\frac{1}{\tau} - \frac{1-a}{q} \right)  
\mathop{=}^{\eqref{balance-law2}} a (d-1) \left( \frac{1}{p} + \frac{\alpha - \sigma - s}{d} \right).  
\end{multline*}
It follows that 
\begin{equation*}
\alpha' - \sigma' = \alpha + \frac{d-1}{p} - \sigma - 
(d-1) \left( \frac{1}{p} + \frac{\alpha - \sigma - s}{d} \right) = \frac{\alpha - \sigma}{d} + \frac{s (d-1)}{d}.  
\end{equation*}
This yields that $\alpha' - \sigma' \ge 0$ if and only if $\alpha - \sigma \ge - s (d-1)$. 

\medskip 
The conclusion now follows from \eqref{thm1-rad-p4} and \eqref{thm1-rad-p5}. \qed

\subsection{Proof of \Cref{thm2-CKNsw}} The proof is in the same spirit of the one of \Cref{thm1-CKNsw}.  For the convenience of the reader, we briefly describe the main lines. Denote $\hu(r) = u(r \sigma)$ with $r > 0$ and $\sigma \in  \mS^{d-1}$. We have, by polar coordinates, 
\begin{equation}
\left(\int_{\ro^d} \frac{|x|^{\gamma \tau}}{\ln^{\mu} (2R_2/|x|)}|u|^{\tau}dx  \right)^{1/ \tau} = |\mS^{d-1}|^{1/\tau} \left( \int_0^\infty \frac{r^{\gamma \tau + d -1}}{\ln^{\mu} (2R_2 r)}|\hu|^{\tau}dr \right)^{1/\tau}
\end{equation}
and 
\begin{equation}
\left(\int_{\ro^d} \frac{|x|^{\gamma \tau}}{\ln^{\mu} (2|x|/R_1)}|u|^{\tau} \, dx  \right)^{1/ \tau} = |\mS^{d-1}|^{1/\tau}  \left(\int_0^\infty \frac{r^{\gamma \tau + (d-1)}}{\ln^{\mu} (2r/R_1)} |\hu|^{\tau} \, dr \right)^{1/ \tau}. 
\end{equation}

Extend $\hu$ in $\mR$ as an even function and still denote the extension by $\hu$. Using \eqref{thm1-rad-m1} and \eqref{thm1-rad-m2}, as in \eqref{thm1-rad-m3}, it suffices to prove that   if  $\mathrm{supp} \ \hu \subset B_{R_2} \subset \mR$,  then it holds 
\begin{equation}
\bct{\int_{\ro} \frac{|\xi|^{(\gamma + \frac{d-1}{\tau} ) \tau }}{\ln^{\mu} (2R_2/|\xi|)} | \hu |^{\tau} \, d \xi}^\frac{1}{\tau} \leq C \nrm{\hu}^a_{\dot{W}^{s, p,\alpha + \frac{d-1}{p}, \Lambda}(\mR)}\nrm{|\xi|^{\beta + \frac{d-1}{q}} |\hu|^{1-a}}_{L^q(\mR)}, 
\end{equation}
and  if $\mathrm{supp} \ \hu \cap B_{R_1}= \emptyset$, then it holds 
\begin{equation}
\bct{\int_{\ro} \frac{|\xi|^{(\gamma + \frac{d-1}{\tau} ) \tau }}{\ln^{\mu} (2|\xi|/ R_1)} |\hu|^{\tau} \, d \xi}^\frac{1}{\tau} \leq C \nrm{\hu}^a_{\dot{W}^{s, p,\alpha + \frac{d-1}{p}, \Lambda}(\mR)}\nrm{|\xi|^{\beta + \frac{d-1}{q}} |\hu|^{1-a}}_{L^q(\mR)}. 
\end{equation}

The conclusion now follows from \Cref{thm2-CKNsw} as in the proof of \Cref{thm2-rad}. The details are omitted. \qed

\medskip
We next show the optimality  of condition $\alpha-\sigma\ge-(d-1)s$ given in \eqref{cond-s-1-*}.

\begin{pro}\label{optimality}
The condition $\alpha-\sigma\geq -(d-1)s$ in \eqref{cond-s-1-*} is necessary for the assertions in \Cref{thm1-rad} to hold.
\end{pro}
\begin{proof}
Let $v\in C_c^\infty(\ro)$ with $\mathrm{supp}\ v \subset (0,1)$. For large $R>0$ define $u_R(x):= v(|x|-R)$, for $x\in \ro^d$. Clearly, $u_R\in C_c^\infty(\ro^d)$ with $\mathrm{supp} \ u_R \subset \mcal{A}_{R,R+1}$, where for any $b, c \in (0,\infty)$, with $b<c$, the set $\mcal{A}_{b,c}$ is defined by $$\mcal{A}_{b,c}:= \{x\in \ro^d: b<|x|<c \}.$$ 
We denote $$\gamma':=\frac{d-1}{\tau}+\gamma, \ \alpha': = \frac{d-1}{p}+\alpha, \ \beta':= \frac{d-1}{q}+\beta.$$ 
One can check that  
 \begin{equation}\label{est1}
\nrm{u_R}_{\dot{W}^{s,p,0,\Lambda}(\ro^d)} \le C \nrm{u_R}_{W^{1,p}(\ro^d)} \le C R^\frac{d-1}{p}, 
 \end{equation}
and  since $\mathrm{supp} \ u_R \subset \mcal{A}_{R,R+1}$, 
 \begin{equation}\label{est1-1}
\nrm{u_R}_{\dot{W}^{s,p,\alpha,\Lambda}(\ro^d)} \le \nrm{u_R}_{\dot{W}^{s,p,\alpha,\Lambda}\bct{\mcal{A}_{R\Lambda^{-1}, (R+1)\Lambda} \times \mcal{A}_{R\Lambda^{-1}, (R+1)\Lambda}}} \le C R^{\alpha } \nrm{u_R}_{\dot{W}^{s,p,0,\Lambda}(\ro^d)}. 
 \end{equation}
Combining \eqref{est1} and \eqref{est1-1} yields  
 \begin{equation}\label{est2}
 \nrm{u_R}_{\dot{W}^{s,p,\alpha,\Lambda}(\ro^d)}^a \le  C R^{a\alpha'}.
 \end{equation}

 On the other hand, one can check that 
 \begin{equation}\label{est3}
 \nrm{|x|^\gamma u_R}_{L^\tau(\ro^d)} \sim R^{\gamma'}  \text{ and } \nrm{|x|^\beta u_R}_{L^q(\ro^d)}^{1-a} \sim R^{(1-a)\beta'}.
 \end{equation}
 Therefore, if either \eqref{thm1-rad-1} or \eqref{thm1-rad-2} holds then using them for $u=u_R$, we conclude form \eqref{est2} and \eqref{est3} 
 $$R^{\gamma'}\leq C R^{a\alpha'+(1-a)\beta'}  \text{ for } R \text{ large and } C>0 \text{ independent of } R,$$
 which is possible only when $\alpha-\sigma \geq -(d-1)s$. The proof is complete. 
 \end{proof}

\section{The Caffarelli-Kohn-Nirenberg inequalities for radial functions in the Sobolev spaces}\label{sect-s=1}

In this section, we present the result in the case $s=1$. We first state variants/improvements of the  CKN inequalities in the Sobolev spaces which follows directly from the approach given in \cite{Ng-Squa1} (see also the proof of \Cref{thm1-CKNsw}). We begin with the case $1/ \tau + \gamma/ d \neq 0$.

\begin{thm}\label{thm3}
Let $d\geq 1$, $p \ge 1$, $q\geq 1$, $\tau \ge 1$, $0 <  a \leq 1$,  and $\alpha, \beta, \gamma \in \ro$. Define $\sigma $ by \eqref{def-alpha-gamma}. Assume \eqref{balance-law1}, \eqref{coucou1}, and \eqref{coucou2}.  We have, for some positive constant $C$, 
\begin{itemize}
\item[$i)$] if $\frac{1}{\tau}+\frac{\gamma}{d}>0$, then for all $u\in L^1_{loc}(\ro^d \setminus \{0 \})$ with compact support in $\mR^d$, it  holds
\begin{equation}\label{CKN-original-M}
\nrm{|x|^\tau u}_{L^\tau(\ro^d)} \leq C\nrm{|x|^\alpha \nabla u}^a_{L^p(\ro^d \setminus \{0 \})} \nrm{|x|^\beta u}^{1-a}_{L^q(\ro^d)}, 
\end{equation}

 \item[$ii)$] if $\frac{1}{\tau}+\frac{\gamma}{d}<0$, then for all $u\in L^1_{loc}(\ro^d \setminus \{0 \})$ which is 0 in a neighborhood of $0$, \eqref{CKN-original-M} holds.  
\end{itemize}
\end{thm}

Concerning the limiting case $1/ \tau +  \gamma/ d = 0 $, one has the following result.  

\begin{thm}\label{thm4}
Let $d\geq 1$, $p \ge 1$, $q\geq 1$, $\tau \ge 1$, $0 <  a \leq 1$,  and $\alpha, \beta, \gamma \in \ro$, and $\mu> 1$. Assume that $\tau \le \mu$.  Define $\sigma $ by \eqref{def-alpha-gamma}. Assume \eqref{balance-law1} and 
\begin{equation}\label{thm4-st1}
0 \le \alpha - \sigma \le 1. 
\end{equation}
  There exists a positive constant $C$ such that for all $u\in L^1_{loc}(\ro^d \setminus \{0 \})$ and for all $R_1, R_2 > 0 $,  we have 
\begin{itemize}
\item[(i)] if $\frac{1}{\tau}+\frac{\gamma}{d}=0$ and $\mathrm{supp}\ u \subset B_{R_2}$, then 
\begin{equation}\label{thm4-p1}
\bct{\int_{\ro^d} \frac{|x|^{\gamma \tau}}{\ln^{\mu} (2R_2/|x|)}|u|^{\tau}dx}^\frac{1}{\tau}\leq C \| |x|^{\alpha} \nabla u \|_{L^p(\mR^d \setminus \{0 \})}^a \nrm{|x|^\beta u}^{1-a}_{L^q(\ro^d)}, 
\end{equation}
 \item[(ii)] if $\frac{1}{\tau}+\frac{\gamma}{d}=0$, and $\mathrm{supp} \ u \cap B_{R_1}= \emptyset$, then 
 \begin{equation}\label{thm4-p2}
\bct{\int_{\ro^d} \frac{|x|^{\gamma \tau}}{\ln^{\mu} (2|x|/R_1)}|u|^{\tau}dx}^\frac{1}{\tau}\leq C \| |x|^{\alpha} \nabla u \|_{L^p(\mR^d \setminus \{0 \})}^a \nrm{|x|^\beta u}^{1-a}_{L^q(\ro^d)}. 
\end{equation}
\end{itemize}
\end{thm}

We are ready to state the corresponding results in the radial case. We begin with the case $1/ \tau + \gamma/ d \neq 0$. 

\begin{thm}\label{thm1-rad-s=1}
Let $d\geq 2$, $p \ge 1$, $q\geq 1$, $\tau \ge 1$, $0 <  a\leq 1$ and $\alpha, \beta, \gamma \in \ro$. Define $\sigma $ by \eqref{def-alpha-gamma}.  Assume \eqref{balance-law1} and 
\begin{equation}\label{thm1-rad-s=1-alpha-gamma}
- (d-1)   \leq \alpha-\sigma < 0.
\end{equation}
We have, for some positive constant $C$, 
\begin{itemize}
\item[$i)$] if $\frac{1}{\tau}+\frac{\gamma}{d}>0$, then for all radial $u\in L^1_{loc}(\ro^d \setminus \{0\})$ with compact support in $\mR^d$, \eqref{CKN-original-M} holds; 
 \item[$ii)$] if $\frac{1}{\tau}+\frac{\gamma}{d}<0$, then for all radial $u\in L^1_{loc}(\ro^d \setminus \{0 \})$ which is 0 in a neighborhood of $0$, \eqref{CKN-original-M} holds.  
\end{itemize}
\end{thm}

Concerning the limiting case $1/ \tau +  \gamma/ d = 0 $, we obtain the following result.  

\begin{thm}\label{thm2-rad-s=1}
Let $d\geq 2$, $p\ge1$, $q\geq 1$, $\tau \ge 1$, $0< a\leq 1$, $\alpha, \beta, \gamma \in \ro$, and $\mu > 1$. Assume that $\tau \le \mu$.  Define $\sigma $ by \eqref{def-alpha-gamma}.  Assume   \eqref{balance-law1} and \eqref{thm1-rad-s=1-alpha-gamma}.
There exists a positive constant $C$ such that for all $u\in L^1_{loc}(\ro^d \setminus \{0 \})$ and for all $0<R_1 < R_2$,  we have 
\begin{itemize}
\item[$i)$] if $\frac{1}{\tau}+\frac{\gamma}{d}=0$ and $\mathrm{supp}\ u \subset B_{R_2}$, then  \eqref{thm4-p1} holds. 

 \item[$ii)$] if $\frac{1}{\tau}+\frac{\gamma}{d}=0$, and $\mathrm{supp} \ u \cap B_{R_1}= \emptyset$, then \eqref{thm4-p2} holds. 
\end{itemize}
\end{thm}

\begin{remark}  \rm The convention in \Cref{rem-convention} is also used in \Cref{thm3}, \Cref{thm4},  \Cref{thm1-rad-s=1}, and \Cref{thm2-rad-s=1}.  In these theorems, the quantity 
$\nrm{|x|^\alpha \nabla u}_{L^p(\ro^d \setminus \{0 \})}$  is  also considered as infinity if $\nabla u \not \in [L^p_{loc}(\mR^d \setminus \{0 \})]^d$. 
\end{remark}

\begin{remark} \rm
By similar considerations as in \Cref{log-term-necessity}, we can conclude that the $\log$-term is necessary in \Cref{thm2-rad-s=1}.
\end{remark}

\Cref{thm1-rad-s=1} and \Cref{thm2-rad-s=1} are direct consequences of \Cref{thm3} and \Cref{thm4} in the one dimensional case. The proofs are  as in the spirit of the proof of \Cref{thm1-rad} and \Cref{thm2-rad} but simpler where a variant of \Cref{lem-rad1} is not required. The details are left to the reader.

\medskip 
Concerning the optimality of the condition of \eqref{thm1-rad-s=1-alpha-gamma}, we have the following result whose proof is similar to the one of  \Cref{optimality} and omitted.

\begin{pro}\label{optimality}
The condition $\alpha-\sigma\geq -(d-1)$ in \eqref{thm1-rad-s=1-alpha-gamma} is necessary for the assertions in \Cref{thm1-rad-s=1} to hold.
\end{pro}

We end this section by mentioning what has been proved previously.  In the case $1/ \tau + \gamma/ d > 0$, under the following additional requirement (see \cite[the first inequality in (1.8) and (1.10)]{DDD12})  
$$
\frac{a (\alpha - 1 - \sigma)}{d} + \frac{1- a}{q} \ge 0 \quad \mbox{ and } \quad \frac{1}{p} + \frac{\alpha - 1}{d} > 0, 
$$
assertion $i)$ of \Cref{thm1-rad-s=1}  was previously proved in \cite{DDD12} by a different approach via the Riesz potential and inequalities on fractional integrations.

\section{Applications to the compactness} \label{sect-compact}

In this section, we derive several compactness results from previous inequalities for radial case. We only consider the case $1 / \tau + \gamma / d > 0$. We begin with the following result. 

\begin{pro}\label{pro-compact-1} Let $d \ge 1$, $0 < s < 1$, $p> 1$, $q \ge 1$,  $\tau \ge 1$, $0< a < 1$, $\alpha,  \, \beta,  \, \gamma \in \mR$, and $\Lambda > 1$ be such that $1 / \tau + \gamma / d > 0$. Define $\sigma $ by \eqref{def-alpha-gamma}. 
Assume  \eqref{balance-law2},   
$$
\alpha - \sigma >  0, \quad \mbox{ and } \quad \frac{1}{p} + \frac{\alpha - s}{d} \neq \frac{1}{q} + \frac{\beta}{d}. 
$$
Assume that the embedding $W^{s, p}(B_1) \cap L^q(B_1)$ into $L^\tau(B_1)$ is compact. 
Let $(u_n)_n \subset L^1_{loc}(\mR^d)$ with compact support be such that the sequences $\big(\| u_n \|_{\dot{W}^{s, p, \alpha, \Lambda}(\mR^d)} \big)_n$ and $\big(\| |x|^\beta u_n\|_{L^q(\mR^d)} \big)_n$ are bounded. Then, up to a subsequence, $(|x|^\gamma u_n)_n$ converges in $L^\tau(\mR^d)$. 
\end{pro}

\begin{proof} One just notes that for $\gamma'$ sufficiently close to $\gamma$, one can choose $0< a' < 1$ close to $a$ such that the assumptions  of \Cref{pro-compact-1} hold with $(a, \gamma)$ being replaced by $(a', \gamma')$. This implies, by \Cref{thm1-CKNsw} (see also \eqref{CKN-fractional1}) that, for $\eps > 0$ sufficiently small, 
$$
(\| |x|^{\gamma + \eps} u_n\|_{L^\tau}) \mbox{ and } (\| |x|^{\gamma - \eps} u_n\|_{L^\tau}) \mbox{ are bounded}.  
$$
The conclusion follows since the embedding $W^{s, p}(B_R) \cap L^q(B_R)$ into $L^\tau(B_R)$ is compact for $R>0$. 
\end{proof}

In the case $s=1$, one has the following result, whose proof is almost identical and  omitted.
 
\begin{pro}\label{pro-compact-1-s=1} Let $d \ge 1$, $p \ge 1$, $q \ge 1$,  $\tau \ge 1$, $0< a < 1$, and  $\alpha, \beta, \gamma \in \mR$ be such that $1 / \tau + \gamma / d > 0$. Define $\sigma $ by \eqref{def-alpha-gamma}. 
Assume \eqref{balance-law1}, 
$$
\alpha - \sigma >  0, \quad \mbox{ and } \quad \frac{1}{p} + \frac{\alpha - 1}{d} \neq \frac{1}{q} + \frac{\beta}{d}. 
$$
Assume that the embedding $W^{1, p}(B_1) \cap L^q(B_1)$ into $L^\tau(B_1)$ is compact. 
Let $(u_n)_n \subset L^1_{loc}(\mR^d)$ with compact support be such that the sequences $\big(\| |x|^\alpha \nabla u_n \|_{L^{p}(\mR^d)} \big)_n$ and $\big(\| |x|^\beta u_n\|_{L^q(\mR^d)} \big)_n$ are bounded. Then, up to a subsequence, $\big(|x|^\gamma u_n \big)_n$ converges in $L^\tau(\mR^d)$. 
\end{pro}

Here are the variants for radial functions, whose proof are almost the same and  omitted. 
 
\begin{pro} \label{pro-compact-2} Let $d \ge 2$, $0 < s < 1$, $p> 1$, $q \ge 1$,  $\tau \ge 1$, $0< a < 1$, $\alpha,  \, \beta, \,  \gamma \in \mR$, and $\Lambda > 1$ be such that $1 / \tau + \gamma / d > 0$. Define $\sigma $ by \eqref{def-alpha-gamma}. 
Assume  \eqref{balance-law2}, 
$$
\alpha - \sigma > - (d-1)s, \quad \mbox{ and } \quad \frac{1}{p} + \frac{\alpha - s}{d} \neq \frac{1}{q} + \frac{\beta}{d}. 
$$
Assume that the embedding $W^{s, p}(B_1) \cap L^q(B_1)$ into $L^\tau(B_1)$ is compact. 
Let $(u_n)_n \subset L^1_{loc}(\mR^d)$ be radial  such that the sequences $\big(\| u_n \|_{\dot{W}^{s, p, \alpha, \Lambda}(\mR^d)} \big)_n$ and $\big(\| |x|^\beta u_n\|_{L^q(\mR^d)} \big)_n$ are bounded. Then, up to a subsequence, $\big(|x|^\gamma u_n\big)_n$ converges in $L^\tau(\mR^d)$. 
\end{pro}

\begin{pro}\label{pro-compact-1-s=1} 
Let $d \ge 2$, $p \ge 1$, $q \ge 1$,  $\tau \ge 1$, $0< a < 1$,  and $\alpha, \beta, \gamma \in \mR$. Define $\sigma $ by \eqref{def-alpha-gamma}.
Assume \eqref{balance-law1},  
$$
\alpha - \sigma >  -(d-1), \quad \mbox{ and } \quad \frac{1}{p} + \frac{\alpha - 1}{d} \neq \frac{1}{q} + \frac{\beta}{d}. 
$$
Assume that the embedding $W^{1, p}(B_1) \cap L^q(B_1)$ into $L^\tau(B_1)$ is compact. 
Let $(u_n)_n \subset L^1_{loc}(\mR^d)$ be radial with compact support  such that the sequences $\big(\| |x|^\alpha \nabla u_n \|_{L^{ p}(\mR^d)} \big)_n$ and $\big(\| |x|^\beta u_n\|_{L^q(\mR^d)} \big)_n$ are bounded. Then, up to a subsequence, $\big(|x|^\gamma u_n \big)_n$ converges in $L^\tau(\mR^d)$. 
\end{pro}

We obtain  the following corollary after using the density of the radial functions $C^\infty_c(\mR^d)$ in the class of radial functions in $W^{s, p}(\mR^d)$. 
 
\begin{cor} Let $d \ge 2$, $0 < s \le 1$, $p \ge 1$ and $sp<d$. Assume that  $ p< \tau < pd/ (d-sp)$ and ($p>1$ if $s < 1$). 
Let $\gamma_1> 0$ and $\gamma_2< 0$ be such that, for $j=1, 2$,  
$$
\frac{1}{p} - \frac{s}{d} < \frac{1}{\tau} + \frac{\gamma_j}{d} < \frac{1}{p}. 
$$
Then the  embedding $W^{s, p}(\mR^d)$ into $L^{\tau}(|x|^{\gamma_j}, \mR^d)$ for radial functions  is compact. As a consequence,  the  embedding $W^{s, p}(\mR^d)$ into $L^{\tau}( \mR^d)$ in the class of radial functions is compact. 
\end{cor}

\begin{remark}
\rm The fact that the embedding $W^{s, p}(\mR^d)$ into $L^{\tau}( \mR^d)$ in the class of radial functions is compact is known, see e.g.,   \cite{Strauss77, Lions82, BL1} in the case $s=1$ and 
\cite{SS00} in the case $0 < s < 1$ (whose proof is based on the atomic decomposition). The ideas to derive the compactness as presented here are quite standard, see,  e.g.,   \cite{BelFrVis}. 
\end{remark}

\bigskip
\noindent{\bf Data statement:} Data sharing is not applicable to this article as no datasets were generated or analysed during the current study.

\bigskip
\noindent{\bf Conflict of interest:} There is no conflict of interest.

\providecommand{\bysame}{\leavevmode\hbox to3em{\hrulefill}\thinspace}
\providecommand{\MR}{\relax\ifhmode\unskip\space\fi MR }
\providecommand{\MRhref}[2]{%
  \href{http://www.ams.org/mathscinet-getitem?mr=#1}{#2}
}
\providecommand{\href}[2]{#2}

\end{document}